\documentclass[10pt,english,a4paper]{amsart}
\pdfoutput=1

\usepackage[utf8]{inputenc}
\usepackage[T1]{fontenc}
\usepackage{lmodern}
\usepackage{babel}
\usepackage{amsmath,amssymb,amsthm}
\usepackage[marginratio={1:1,1:1},totalwidth=400pt,totalheight=600pt]{geometry}
\usepackage{microtype}
\usepackage{hyperref}
\usepackage{url}
\usepackage{tikz-cd}

\theoremstyle{plain}
\newtheorem{theorem}{Theorem}[section]
\newtheorem{lemma}[theorem]{Lemma}
\newtheorem{proposition}[theorem]{Proposition}
\newtheorem{corollary}[theorem]{Corollary}

\theoremstyle{definition}

\theoremstyle{remark}

\newcommand{\C}{\mathbb{C}}

\newcommand{\Z}{\mathbb{Z}}

\newcommand{\T}{\mathbb{T}}

\title{Free~nilpotent~groups are $C^*$-superrigid}

\author[Omland]{Tron Omland}
\address{Department of Mathematics\\University of Oslo\\NO-0316 Oslo\\Norway
\and
Department of Computer Science\\Oslo Metropolitan University\\NO-0130 Oslo\\Norway}
\email{trono@math.uio.no}

\date{April 24, 2019}

\thanks{The author is funded by the Research Council of Norway through FRINATEK, project no.~240913.}

\begin{document}

\begin{abstract}
The free nilpotent group $G_{m,n}$ of class $m$ and rank $n$ is the free object on $n$ generators in the category of nilpotent groups of class at most $m$.
We show that $G_{m,n}$ can be recovered from its reduced group $C^*$-algebra,
in the sense that if $H$ is any group such that $C^*_r(H)$ is isomorphic to $C^*_r(G_{m,n})$, then $H$ must be isomorphic to $G_{m,n}$.
\end{abstract}

\maketitle

\section*{Introduction}

Group $C^*$-algebras play an important role in the theory of operator algebras.
A natural question to ask, yet not much studied, is to what extent a group can be recovered from its (reduced) group $C^*$-algebra.
The analog problem for group von~Neumann algebras has received some attention in the last few years,
but to this day there are less than a handful of results available, the first one presented in \cite{IPV}.
A group $G$ is called $W^*$-superrigid if it can be recovered from its group von~Neumann algebra $L(G)$,
that is, if $H$ is any group such that $L(H)\cong L(G)$, then $H\cong G$.
The group von~Neumann algebra of any nontrivial countable amenable group with infinite conjugacy classes is isomorphic to the hyperfinite $\textup{II}_1$ factor,
so in general, much of the group structure is lost in the construction.
However, examples of $W^*$-superrigid groups are known to exist, in particular, some classes of generalized wreath products \cite{IPV,BV} and amalgamated free products \cite{CI}.


Inspired by this terminology, a group $G$ is said to be $C^*$-superrigid if $C^*_r(H)\cong C^*_r(G)$ implies that $H\cong G$.
It has been known for some time that torsion-free abelian groups are $C^*$-superrigid \cite{Sch},
and only very recently, it was shown that certain torsion-free virtually abelian groups, so-called Bieberbach groups, are $C^*$-superrigid \cite{KRTW},
providing the first result for nonabelian groups.
In a somewhat different direction, specific examples of amalgamated free products were proven to be $C^*$-superrigid in \cite{CI},
including a continuum of groups that can contain torsion.
Returning to the amenable situation,
it is conjectured that all finitely generated torsion-free nilpotent groups are $C^*$-superrigid,
and important progress towards solving this problem was made in \cite{ER},
where the authors gave a positive answer in the case of nilpotency class $2$.



We remark that there is no known example of a torsion-free group that is not $C^*$-superrigid.
For more background on the topic, see \cite{KRTW,ER} and references therein.

In this short note, we show that also the free nilpotent groups are $C^*$-superrigid.

\section{Preliminaries and various results}

Let $G$ be a discrete group.
As usual, $C^*(G)$ denotes the full group $C^*$-algebra of $G$,
and we let $g\mapsto u_g$ be the canonical inclusion of $G$ into $C^*(G)$.
The left regular representation $\lambda$ of $G$ on $\ell^2(G)$ is given by $\lambda_g\delta_h=\delta_{gh}$ for all $g,h\in G$,
and the reduced group $C^*$-algebra $C^*_r(G)$ of $G$ is the $C^*$-subalgebra of $B(\ell^2(G))$ generated by the image of $\lambda$.
It follows that $\lambda$ induces a homomorphism of $C^*(G)$ onto $C^*_r(G)$, mapping $u_g$ to $\lambda_g$ for all $g\in G$.
Moreover, it is well-known that if $C^*(G)\cong C^*_r(G)$, then $\lambda$ must be faithful, and in this case, $G$ is called amenable, and we use $\lambda$ to identify $C^*(G)$ with $C^*_r(G)$. 

The subgroup $G'$ of $G$ generated by all the elements $ghg^{-1}h^{-1}$ for $g,h\in G$ is called the commutator (or derived) subgroup of $G$. 
It is normal in $G$, and the quotient $G_\textup{ab}=G/G'$ is an abelian group, called the abelianization of $G$.
The group $G_\textup{ab}$ is the largest abelian quotient of $G$, that is, whenever $N$ is a normal subgroup of $G$ and $G/N$ is abelian, $G'\subseteq N$.

Let $\widetilde{\pi}_\textup{ab}\colon C^*(G)\to C^*(G_\textup{ab})$ denote the homomorphism induced by the quotient map $\pi_\textup{ab}\colon G\to G_\textup{ab}$.
Note that $\pi_\textup{ab}$ induces a map $C^*_r(G)\to C^*_r(G_\textup{ab})=C^*(G_\textup{ab})$ if and only if $G'$, or equivalently, $G$ is amenable. 

For a $C^*$-algebra $A$, the commutator ideal $\mathcal{J}$ of $A$ is the ideal generated by all elements $xy-yx$ for $x,y\in A$.
Let $\phi\colon A\to A/\mathcal{J}$ denote the quotient map.
The Gelfand spectrum $\Gamma_A$ of $A$ is given by
\[
\Gamma_A=\{ \text{nonzero algebra homomorphisms } \gamma\colon A\to\C \}.
\]
For every $\gamma\in\Gamma_A$, we clearly have $\gamma(xy-yx)=0$ for all $x,y\in A$, and thus $\mathcal{J}\subseteq\ker\gamma$.
If $\rho\in\Gamma_{A/\mathcal{J}}$, then $\rho\circ\phi$ belongs to $\Gamma_A$,
and every $\gamma\in\Gamma_A$ defines an element $\rho\in\Gamma_{A/\mathcal{J}}$ given by $\rho(x+\mathcal{J})=\gamma(x)$.
Together, this gives that $\Gamma_{A/\mathcal{J}}=\Gamma_A$. 
Moreover, if $x\notin\mathcal{J}$, then $0\neq\phi(x)\in A/\mathcal{J}$, which is a commutative $C^*$-algebra, so there exists $\rho\in\Gamma_{A/\mathcal{J}}$ such that $\rho(\phi(x))\neq 0$.
That is, $x\notin\ker\rho\circ\phi$, and as explained above, $\rho\circ\phi\in\Gamma_A$.
We conclude that
\begin{equation}\label{intersection}
\mathcal{J}=\bigcap_{\gamma\in\Gamma_A}\ker\gamma.
\end{equation}


\begin{lemma}
The commutator ideal $\mathcal{J}$ of $C^*(G)$ coincides with the kernel of $\widetilde{\pi}_\textup{ab}$.
\end{lemma}

\begin{proof}
First, since $C^*(G_\textup{ab})$ is commutative, $\ker\widetilde{\pi}_\textup{ab}$ must contain all commutators in $C^*(G)$, and thus $\mathcal{J}\subseteq\ker\widetilde{\pi}_\textup{ab}$.
Next, we note that
\[
\Gamma_{C^*(G_\textup{ab})}=\operatorname{Hom}(G_\textup{ab},\T)=\operatorname{Hom}(G,\T)=\Gamma_{C^*(G)}. 
\]
The second identification is given by $\chi'\mapsto\chi'\circ\pi_\textup{ab}$ for $\chi'\in\operatorname{Hom}(G_\textup{ab},\T)$,
and the inverse by $\chi\mapsto\chi'$ for $\chi\in\operatorname{Hom}(G,\T)$, where $\chi'(g+G')=\chi(g)$.
The last identification is the usual integrated form,
with inverse $\gamma\mapsto\chi$ for $\gamma\in\Gamma_{C^*(G)}$, where $\chi(g)=\gamma(u_g)$;
and the first equality is similar.
Combined, the first and last space is identified via $\gamma'\mapsto\gamma'\circ\widetilde{\pi}_\textup{ab}$ for $\gamma'\in\Gamma_{C^*(G_\textup{ab})}$.

Thus, if $x\notin\mathcal{J}$, then by \eqref{intersection} there is $\gamma\in\Gamma_{C^*(G)}$ such that $\gamma(x)\neq 0$.
Since $\gamma=\gamma'\circ\widetilde{\pi}_\textup{ab}$ for some $\gamma'\in\Gamma_{C^*(G_\textup{ab})}$,
we have $\gamma'(\widetilde{\pi}_\textup{ab}(x))\neq 0$, and hence $x\notin\ker\widetilde{\pi}_\textup{ab}$.
\end{proof}

The following result is due to \cite{Sch}. 
\begin{proposition}\label{folklore}
Suppose that $G$ is torsion-free and abelian and let $H$ be any group such that $C^*(H)\cong C^*(G)$.
Then $H\cong G$.
\end{proposition}

\begin{corollary}\label{commutator}
If $H$ is any group such that $C^*(H)\cong C^*(G)$, then $C^*(H_\textup{ab})\cong C^*(G_\textup{ab})$.
In particular, if $G_\textup{ab}$ is torsion-free, then $H_\textup{ab}\cong G_\textup{ab}$.
\end{corollary}

\begin{proof}
Any isomorphism $C^*(H)\cong C^*(G)$ takes the commutator ideal of $C^*(H)$ to the commutator ideal of $C^*(G)$,
and thus, the quotients $C^*(H_\textup{ab})$ and $C^*(G_\textup{ab})$ must be isomorphic.
The last statement now follows from Proposition~\ref{folklore}.
\end{proof}

The upper central sequence of $G$, denoted $Z_0 \subset Z_1 \subset Z_2 \subset \dotsb$,
is defined by $Z_0=\{e\}$, $Z_1=Z(G)$, and for all $i\geq 0$,
\[
Z_{i+1}=\{g\in G:[g,h]\in Z_i\text{ for all }h\in G\}.
\]
In particular, we remark that $Z_i$ is a normal subgroup of $Z_{i+1}$ and $Z_{i+1}/Z_i=Z(G/Z_i)$ for all $i\geq 0$.
If there exists an $m$ such that $G=Z_m$, then $G$ is called a nilpotent group,
and the smallest such $m$ is said to be the class of $G$.


\begin{lemma}\label{generating}
Suppose that $G$ is a nilpotent group and let $S\subseteq G$ be a set such that $\pi_\textup{ab}(S)$ generates $G_\textup{ab}$.
Then $S$ generates $G$.
\end{lemma}

\begin{proof}
Let $m$ be the nilpotency class of $G$, and let $\{e\}=Z_0\subset Z_1\subset\dotsb\subset Z_{m-1}\subset Z_m=G$ be the upper central series of $G$.
Denote by $H$ the subgroup of $G$ generated by $S$.
For $0\leq i\leq m$, set $H_i=HZ_i$.
Then $H_i$ is a subgroup of $G$ and a normal subgroup of $H_{i+1}$ for all $i$.
Indeed, for $h,h'\in H$, $z_i,z_i'\in Z_i$,
\[
(h'z_i)(hz_i')=h'hz_i[z_i^{-1},h^{-1}]z_i' \in HZ_iZ_{i-1}Z_i=HZ_i=H_i
\]
since $[z_i^{-1},h^{-1}]\in Z_{i-1}$.
Moreover, for $h,h'\in H$, $z_i,z_i'\in Z_i$, and $z_{i+1}\in Z_{i+1}$,
\[
(hz_{i+1})(h'z_i)(hz_{i+1})^{-1}=h[z_{i+1},h']h'z_i[z_i^{-1},z_{i+1}]h^{-1} \in HZ_iHZ_iZ_{i-1}H=H_i.
\]
If $H\neq G$, there would exist some $0\leq k<m$ such that $H_k\neq G$ and $H_{k+1}=G$.
Then
\[
G/H_k=H_{k+1}/H_k=HZ_{k+1}/HZ_k\cong Z_{k+1}/Z_k,
\]
where the last identification is the second isomorphism theorem, and the last quotient is abelian. 
Thus, $H_k$ contains the commutator subgroup $G'$, and therefore also $HG'$.
Since $\pi_\textup{ab}(H)=G_\textup{ab}$, then $HG'$ equals $G$, so 
\[
G=HG'\subseteq H_k \subsetneq G,
\]
which is a contradiction.
Hence, we conclude that $H=G$.
\end{proof}

Note that the abelianization of a torsion-free nilpotent group is not necessarily torsion-free,
so in general, we do not know if it can be recovered from its group $C^*$-algebra.
E.g., consider the index~$2$ subgroup of the Heisenberg group given by
\[
G=
\begin{pmatrix}
1 & 2\Z & \Z \\
0 & 1 & \Z \\
0 & 0 & 1
\end{pmatrix}.
\]
The abelianization of $G$ is $\Z^2\oplus(\Z/2\Z)$, and all generating sets of $G$ have at least three elements.
This is in constrast to the Heisenberg group itself, whose abelianization is $\Z^2$, and which can be generated by two elements.

\section{\texorpdfstring{$C^*$}{C*}-superrigidity of free nilpotent groups}

The free nilpotent group $G_{m,n}$ of class $m$ and rank $n$ is the free object on $n$ generators in the category of nilpotent groups of class at most $m$.
It is defined by a set of generators $\{g_i\}_{i=1}^n$ subject to the relations that all commutators of length $m+1$ involving the generators are trivial,
i.e., $[\dotsb[[g_{i_1},g_{i_2}],g_{i_3}]\dotsb],g_{i_{m+1}}]$ is trivial for any choice of sequence of generators.
For all $m\geq 1$, we have $G_{m,1}\cong\Z$, while $G_{m,n}$ is an $m$-step nilpotent group for every $n\geq 2$.
As an easy example, we mention that $G_{2,2}$ is isomorphic to the Heisenberg group,
and refer to \cite{Tao,Om} for further details about free nilpotent groups.

The group $G_{m,n}$ satisfies the following universal property:
If $H$ is any nilpotent group of class at most $m$ and $h_1,\dotsc,h_n$ are elements in $H$,
there exists a unique homomorphism $G_{m,n}\to H$ mapping $g_i$ to $h_i$ for all $i$.

The abelianization of $G_{m,n}$ is isomorphic to $\Z^n$ and $\pi_\textup{ab}$ maps $g_i$ to the generator $e_i$ of the i'th summand of $\Z^n$.

The center $Z(G_{m,n})$ of $G_{m,n}$ is a free abelian group (its rank can be computed, but it is not relevant here),
and for $m,n\geq 2$ we have
\begin{equation}\label{class-1}
G_{m,n}/Z(G_{m,n}) \cong G_{m-1,n},
\end{equation}
as seen by mapping generators to generators. 

\begin{lemma}\label{characterization}
Let $m,n\geq 2$, and let $H$ be a nilpotent group of class at most $m$ that can be generated by $n$ elements.
Suppose that $H/Z(H) \cong G_{m-1,n}$.
Then $H\cong G_{m,n}$.
\end{lemma}

\begin{proof}
The universal property of $G_{m,n}$ means that there exists a surjective map $\varphi\colon G_{m,n}\to H$.
Clearly, $\varphi(Z(G_{m,n}))\subseteq Z(H)$, and we set $K=\varphi^{-1}(Z(H))$.
Consider the maps
\[
G_{m,n}/Z(G_{m,n}) \to G_{m,n}/K \to H/Z(H),
\]
given by $aZ(G_{m,n}) \mapsto aK$ and $aK \mapsto \varphi(a)Z(H)$.
The composition map $\psi$ is surjective since $\varphi$ is surjective.
Since finitely generated nilpotent groups are Hopfian, $G_{m-1,n}\cong G_{m,n}/Z(G_{m,n})$ does not have any proper quotient isomorphic to itself. 
Hence, the composition map $\psi$ must be an isomorphism, and $K=Z(G_{m,n})$.
We get the following commutative diagram
\begin{displaymath}
\begin{tikzcd}[column sep=large,row sep=large]
1 \arrow{r} & Z(G_{m,n}) \arrow{r}{i}\arrow{d}{\cong}[swap]{\varphi_{|Z}} & G_{m,n} \arrow{r}{q}\arrow{d}[swap]{\varphi} & G_{m,n}/Z(G_{m,n}) \arrow{r}\arrow{d}{\cong}[swap]{\psi} & 1 \\
1 \arrow{r} & Z(H) \arrow{r}{i} & H \arrow{r}{q} & H/Z(H) \arrow{r} & 1 
\end{tikzcd}
\end{displaymath}
By the five lemma, $\varphi$ is an isomorphism.
\end{proof}

\begin{theorem}
For every pair of natural numbers $m$ and $n$,
the free nilpotent group $G_{m,n}$ of class $m$ and rank $n$ is $C^*$-superrigid.
\end{theorem}

\begin{proof}
The case $n=1$ is obvious, so let $n\geq 2$.
We do this by induction on $m$.
Note first that $G_{1,n}\cong\Z^n$, which is $C^*$-superrigid (see Proposition~\ref{folklore}).
Let $m\geq 2$, and suppose that $G_{m-1,n}$ is $C^*$-superrigid.
Let $H$ be any group and assume that $C^*(H) \cong C^*(G_{m,n})$.
It follows from \cite[Theorem~B]{ER} that $H$ is a torsion-free nilpotent group of class $m$. 

Moreover, $C^*(H/Z(H)) \cong C^*(G_{m,n}/Z(G_{m,n}))$ by \cite[Proof of Lemma~4.2]{ER}, 
and \eqref{class-1} implies that the latter is isomorphic to $C^*(G_{m-1,n})$.
By the induction hypothesis, the group $G_{m-1,n}$ is $C^*$-superrigid, so $H/Z(H) \cong G_{m-1,n}$.

The abelianization of $G_{m,n}$ is isomorphic to $\Z^n$ and thus $H_\textup{ab}\cong\Z^n$ by Corollary~\ref{commutator}.
For each $1\leq i\leq n$, choose an element $s_i$ of $H$ that is mapped to the generator $e_i$ of $\Z^n\cong H_\textup{ab}$.
If $S=\{s_i:1\leq i\leq n\}$, then $\pi_\textup{ab}(S)$ generates $H_\textup{ab}$, so $S$ generates $H$ by Lemma~\ref{generating},
i.e., $H$ can be generated by $n$ elements.

Therefore, we apply Lemma~\ref{characterization} to conclude that $H\cong G_{m,n}$.
\end{proof}



\end{document}